\documentclass[a4paper,11pt]{article}
\pdfoutput=1

\usepackage{graphicx, mathtools, amssymb, latexsym, amsmath, amsfonts, amsthm, bbm, tikz}
\usepackage[toc,page]{appendix}
\usepackage[ruled, linesnumbered, noend]{algorithm2e}
\usepackage{enumitem}

\usepackage[T1]{fontenc}

\usepackage{array, fancyhdr, bm, listings, soul, xcolor,titlesec, cancel}
\usepackage[skip=0.5\baselineskip]{caption}

\usepackage{thmtools, thm-restate}
\usepackage{thm-autoref}

\definecolor{Darkblue}{rgb}{0,0,0.4}
\definecolor{Brown}{cmyk}{0,0.61,1.,0.60}
\definecolor{Purple}{cmyk}{0.45,0.86,0,0}
\definecolor{Darkgreen}{rgb}{0.133,0.543,0.133}
\usepackage[colorlinks,linkcolor=Darkblue,filecolor=blue,citecolor=blue,urlcolor=Darkblue,pagebackref]{hyperref}

\newcommand{\OO}{\mathcal{O}}
\newcommand{\spc}{\delta}

\theoremstyle{definition}
\newtheorem{theorem}{Theorem}

\newtheorem{lemma}{Lemma}

\title{Grid Induced Minor Theorem for Graphs of Small Degree}
\author{Tuukka Korhonen\thanks{Department of Informatics, University of Bergen, Norway. \texttt{tuukka.korhonen@uib.no}. This work was supported by the Research Council of Norway via the project BWCA (grant no. 314528).}}
\date{}

\begin{document}
\maketitle

\begin{abstract}
A graph $H$ is an induced minor of a graph $G$ if $H$ can be obtained from $G$ by vertex deletions and edge contractions.
We show that there is a function $f(k, d) = \OO(k^{10} + 2^{d^5})$ so that if a graph has treewidth at least $f(k, d)$ and maximum degree at most $d$, then it contains a \mbox{$k \times k$-grid} as an induced minor.
This proves the conjecture of Aboulker, Adler, Kim, Sintiari, and Trotignon~[Eur.~J.~Comb.,~98,~2021] that any graph with large treewidth and bounded maximum degree contains a large wall or the line graph of a large wall as an induced subgraph.
It also implies that for any fixed planar graph $H$, there is a subexponential time algorithm for maximum weight independent set on $H$-induced-minor-free graphs.
\end{abstract}

\section{Introduction}
A graph $H$ is a minor of a graph $G$ if $H$ can be obtained as a contraction of a subgraph of $G$.
An induced minor is a minor that is obtained as a contraction of an induced subgraph.

The famous grid minor theorem of Robertson and Seymour~\cite{DBLP:journals/jct/RobertsonS86} (see also~\cite{DBLP:journals/jacm/ChekuriC16,DBLP:journals/jctb/ChuzhoyT21}) states that there is a function $f : \mathbb{N} \rightarrow \mathbb{N}$ so that any graph with treewidth at least $f(k)$ contains a $k \times k$-grid as a minor.
A question with several applications in graph theory and algorithms~\cite{DBLP:journals/ejc/AboulkerAKST21,DBLP:journals/jctb/AbrishamiCV22,DBLP:conf/icalp/Korhonen21} is when can ``minor'' be replaced by ``induced minor'' in the grid minor theorem.
In general this is not possible: complete graphs have unbounded treewidth but do not contain a $2 \times 2$-grid as an induced minor.
However, one could expect that minors and induced minors behave similarly in sparse graphs.
Fomin, Golovach, and Thilikos proved that for any fixed graph $H$, there is a constant $c_H$ so that any $H$-minor-free graph with treewidth at least $c_H \cdot k$ contains a $k \times k$-grid as an induced minor~\cite{DBLP:journals/jct/FominGT11}.
In this paper, we give a grid induced minor theorem for another important class of sparse graphs, in particular for the class of graphs with bounded maximum degree.

\begin{theorem}
\label{the:main}
There is a function $f(k, d) = \OO(k^{10} + 2^{d^5})$ so that for any positive integers $k,d$ it holds that if a graph has treewidth at least $f(k, d)$ and maximum degree at most $d$, then it contains a $k \times k$-grid as an induced minor.
\end{theorem}

For graphs with treewidth at least $2^{d^5}$, the size of the grid that we obtain is up to a subpolynomial $2^{\OO(\log^{5/6} k)}$ factor the same as the size of the grid in the grid minor theorem.
In particular, the bound $\OO(k^{10} + 2^{d^5})$ on $f(k, d)$ follows from the most recent bound on the grid minor theorem~\cite{DBLP:journals/jctb/ChuzhoyT21} and will improve if the bound on the grid minor theorem is improved.

Together with some routing arguments from~\cite{DBLP:journals/ejc/AboulkerAKST21}, \autoref{the:main} implies that the following conjecture of Aboulker, Adler, Kim, Sintiari, and Trotignon~\cite{DBLP:journals/ejc/AboulkerAKST21} holds.

\begin{restatable}{corollary}{cormain}
\label{cor:main}
For every $d \in \mathbb{N}$, there is a function $f_d : \mathbb{N} \rightarrow \mathbb{N}$ such that every graph with maximum degree at most $d$ and treewidth at least $f_d(k)$ contains a $k \times k$-wall or the line graph of a $k \times k$-wall as an induced subgraph.
 \end{restatable}

In addition to~\cite{DBLP:journals/ejc/AboulkerAKST21}, this conjecture has been explicitly mentioned by Abrishami, Chudnovsky, Dibek, Hajebi, Rz\k{a}\.{z}ewski, Spirkl, and Vu{\v{s}}kovi{\'c} in three articles of their ``Induced subgraphs and tree decompositions'' series~\cite{abrishami2021induced,abrishami2021induced3,abrishami2022induced} and the main results of the first two articles of this series~\cite{abrishami2021induced,DBLP:journals/jctb/AbrishamiCV22} are special cases of this conjecture.

\autoref{the:main} has several direct algorithmic implications.
In particular, by known algorithms using treewidth~\cite{DBLP:journals/siamcomp/Bodlaender96,DBLP:journals/iandc/Courcelle90}, it implies that a large number of combinatorial problems can be solved in linear-time on graphs that exclude a planar graph as an induced minor and have bounded maximum degree.
Note that in contrast, for any non-planar graph $H$, for example the maximum independent set problem is NP-complete on subcubic graphs that exclude $H$ as an induced minor~\cite{DBLP:journals/jct/Mohar01}.

\autoref{the:main} implies also the following algorithmic result.

\begin{restatable}{corollary}{coralg}
\label{cor:alg}
For every fixed planar graph $H$, there is a $2^{\OO(n / \log^{1/6} n)}$ time algorithm for maximum weight independent set on $H$-induced-minor-free graphs, where $n$ is the number of vertices of the input graph.
\end{restatable}

Algorithms for maximum weight independent set on $H$-induced-minor-free graphs for specific planar graphs $H$ have recently received attention~\cite{DBLP:journals/corr/abs-2111-04543,DBLP:conf/focs/GartlandL20,DBLP:conf/stoc/GartlandLPPR21}.
In this area a central open question asked by Dallard, Milanic, and Storgel~\cite{DBLP:journals/corr/abs-2111-04543} is if there exists a fixed planar graph $H$ so that maximum weight independent set is NP-complete on $H$-induced-minor-free graphs, or whether a polynomial time or quasipolynomial time algorithm could be obtained for $H$-induced-minor-free graphs for all planar graphs $H$.
\autoref{the:main} can be seen as a step towards the latter direction.

\section{Preliminaries}
We use $\log$ to denote base-2 logarithm.
A graph $G$ has a set of vertices $V(G)$ and a set of edges $E(G)$.
For a vertex $v \in V(G)$, $N(v)$ denotes the set of its neighbors.
For a set of vertices $S \subseteq V(G)$, $G[S]$ denotes the subgraph of $G$ induced by $S$.
The distance between two vertices in a graph is the minimum number of edges on a path between them, or if they are in different connected components the distance is infinite.

A minor model of a graph $H$ in a graph $G$ is a collection $\{X_v\}_{v \in V(H)}$ of pairwise disjoint vertex sets $X_v \subseteq V(G)$ called \emph{branch sets} so that each induced subgraph $G[X_v]$ is connected, and if there is an edge $uv \in E(H)$, then there is an edge between $X_u$ and $X_v$ in $G$.
A graph $G$ contains $H$ as a minor if and only if there is a minor model of $H$ in $G$.
An induced minor model of $H$ is a minor model with an additional constraint that if $u,v \in V(H)$, $u \neq v$, and $uv \notin E(H)$, then there are no edges between $X_u$ and $X_v$.
A graph $G$ contains $H$ as an induced minor if and only if there is an induced minor model of $H$ in $G$.

A tree decomposition of a graph $G$ is a pair $(T, \beta)$, where $T$ is a tree and $\beta$ is a function $\beta : V(T) \rightarrow 2^{V(G)}$ mapping nodes of $T$ to sets of vertices of $G$ called bags so that
\begin{enumerate}
\item $V(G) = \bigcup_{i \in V(T)} \beta(i)$,
\item for every $uv \in E(G)$ there exists $i \in V(T)$ with $\{u,v\} \subseteq \beta(i)$, and
\item for every $v \in V(G)$, the nodes $\{i \in V(T) \mid v \in \beta(i)\}$ induce a connected subtree of $T$.
\end{enumerate}
The width of a tree decomposition is $\max_{i \in V(T)} |\beta(i)|-1$ and the treewidth of a graph is the minimum width of a tree decomposition of it.

We recall a standard lemma on treewidth.

\begin{lemma}
\label{lem:disjcontract}
Let $G$ be a graph of treewidth $k$ and let $X_1, \ldots, X_p$ be pairwise disjoint subsets of $V(G)$ so that for each $X_i$ it holds that $|X_i| \le q$ and $G[X_i]$ is connected.
The treewidth of the graph obtained by contracting each set $X_i$ into a single vertex is at least $k/q$.
\end{lemma}
\begin{proof}
Let $G'$ be the graph obtained from $G$ by contracting each set $X_i$ into a vertex $x_i$.
Suppose that $G'$ has a tree decomposition of width at most $k/q-1$.
We construct a tree decomposition of $G$ by replacing each vertex $x_i$ in the decomposition by the set $X_i$.
This increases the bag sizes by a factor of $q$, so we obtain a tree decomposition of $G$ of width at most $k-1$, which is a contradiction.
\end{proof}

\section{Proof of \autoref{the:main}}
We will first define \emph{sparsifiable graphs} and show that minors in sparsifiable graphs correspond to induced minors.
Then we show that if a graph has treewidth $k$ and maximum degree at most $\log^{1/5} k$, then it has an induced subgraph that is sparsifiable and has treewidth $k / 2^{\OO(\log^{5/6} k)}$.

\subsection{Sparsifiable graphs}
We say that a vertex $v$ of a graph is \emph{sparsifiable} if it satisfies one of the following conditions:

\begin{enumerate}
\item $v$ has degree at most 2,
\item $v$ has degree 3 and all of its neighbors have degree at most 2, or
\item $v$ has degree 3, one of its neighbors has degree at most 2, and the two other neighbors form a triangle with $v$.
\end{enumerate}

We call a graph sparsifiable if every vertex of it is sparsifiable.
We say that a vertex is of type~2 if it satisfies the condition~2 and that a vertex is of type~3 if it satisfies the condition~3 but does not satisfy the condition~2.

Now we show that minors and induced minors are highly related in sparsifiable graphs.

\begin{lemma}
\label{lem:spars}
Let $G$ be a sparsifiable graph and $H$ a graph of minimum degree at least 3.
If $G$ contains $H$ as a minor, then $G$ contains $H$ as an induced minor.
\end{lemma}
\begin{proof}
Let $\{X_v\}_{v \in V(H)}$ be a minor model of $H$ in $G$.
Say that an edge $ab \in E(G)$ is violating if there exists non-adjacent vertices $u,v$ of $H$ so that $a \in X_u$ and $b \in X_v$.
Note that if there are no violating edges, then $\{X_v\}_{v \in V(H)}$ is also an induced minor model of $H$.
Suppose that $\{X_v\}_{v \in V(H)}$ minimizes the number of violating edges among all minor models of $H$.
We will show by contradiction that the number of violating edges must be zero.

Let $ab \in E(G)$ be a violating edge with $a \in X_u$ and $b \in X_v$.
Both $a$ and $b$ must have degree at least~2 because otherwise $u$ or $v$ would have degree~1.
First, consider the case that $ab$ has a degree~2 endpoint, and by symmetry assume that the degree~2 endpoint is $a \in X_u$.
The neighbor of $a$ not in $X_v$ must be in $X_u$ because $u$ has degree at least~3.
It also holds that $G[X_u \setminus \{a\}]$ is connected because $a$ has degree 1 in $G[X_u]$.
The only other branch set that $a$ is adjacent to than $X_u$ is $X_v$, but there is no edge between $u$ and $v$ in $H$, so we can replace $X_u$ by $X_u \setminus \{a\}$ in the minor model.
This decreases the number of violating edges.

The other case is that both $a$ and $b$ have degree 3.
Because they are adjacent, they both are of type 3, so they form a triangle together with a vertex $c$.
First, consider the case that $c \in X_v$.
Now, $a$ is adjacent to two vertices in $X_v$ and the third vertex it is adjacent to must be in $X_u$ because $u$ has degree at least~3.
Therefore, we can replace $X_u$ by $X_u \setminus \{a\}$ in the minor model, which decreases the number of violating edges.

The same argument as above works also if $c \in X_u$ or if $c$ is not in any branch set.
The remaining case is that $c \in X_w$ for some $w \in V(H) \setminus \{u,v\}$.
If $w$ is not adjacent to $u$ (resp. to $v$), then removing $a$ from $X_u$ (resp. $b$ from $X_v$) again decreases the number of violating edges, so we can assume that $w$ is adjacent to both $u$ and $v$.
Because $u$ and $v$ have degree at least 3 and $u$ and $v$ are not adjacent, it holds that the unique vertex in $N(a) \setminus \{b,c\}$ is in $X_u$ and that the unique vertex in $N(b) \setminus \{a,c\}$ is in $X_v$.
We replace $X_u$ by $X_u \setminus \{a\}$, $X_v$ by $X_v \setminus \{b\}$, and $X_w$ by $X_w \cup \{a,b\}$.
This removes the violating edge $ab$, and does not create any new violating edges because $w$ is adjacent to both $u$ and $v$, and both $N(a)$ and $N(b)$ are contained in $X_u \cup X_v \cup X_w$.
\end{proof}

\autoref{lem:spars} cannot be directly used when $H$ is the $k \times k$-grid because its corners have degree 2.
However, contracting four edges each incident to a distinct corner of the $k \times k$-grid yields a graph of minimum degree 3 that contains the $k-2 \times k-2$-grid as an induced minor.
Therefore \autoref{lem:spars} implies that if a sparsifiable graph contains a $k \times k$-grid minor, then it contains a $k-2 \times k-2$-grid induced minor.

\subsection{Sparsifying a graph}
We will make use of the following theorem proved by Chekuri and Chuzhoy that every graph contains a degree-3 subgraph that approximately preserves its treewidth.

\begin{theorem}[\cite{DBLP:conf/soda/ChekuriC15}]
\label{the:deg3spa}
There exists a constant $\spc$ so that every graph with treewidth $k \ge 2$ has a subgraph with maximum degree 3 and treewidth at least $k / \log^\spc k$.
\end{theorem}

For the rest of this section we will use $\spc$ to denote the constant $\spc$ given by \autoref{the:deg3spa}.
We note that instead of \autoref{the:deg3spa} we could alternatively use the grid minor theorem, but that would yield a significantly worse dependence on the treewidth in our result.

A distance-5 independent set in a graph $G$ is a set $I \subseteq V(G)$ of vertices so that for any pair $u,v \in I$ of distinct vertices, the distance between $u$ and $v$ in $G$ is at least 5.
Next we show how to make all vertices in a distance-5 independent set sparsifiable while approximately preserving treewidth.
This will be then used to prove \autoref{the:main} by observing that the vertices of a graph with maximum degree $d$ can be partitioned into $d^4+1$ distance-5 independent sets.

\begin{lemma}
\label{lem:main}
Let $G$ be a graph of treewidth $k$ and maximum degree $d$ with $2d^2 \le k$, and let $I \subseteq V(G)$ be a distance-5 independent set in $G$.
There exists an induced subgraph $G[S]$ of $G$ so that $G[S]$ has treewidth at least $k/((d^2+1) \log^\spc k)$ and every vertex in $I \cap S$ is sparsifiable in $G[S]$.
\end{lemma}
\begin{proof}
For each vertex $v \in I$, let $B_v$ be the set of vertices at distance at most 2 from $v$ (i.e. $B_v = \{v\} \cup N(v) \cup N(N(v))$).
The induced subgraphs $G[B_v]$ are connected, and because $I$ is a distance-5 independent set, the sets $B_v$ are disjoint.
Let $H$ be the graph obtained by contracting each set $B_v$ into one vertex.
Because $|B_v| \le d^2+1$, \autoref{lem:disjcontract} implies that the treewidth of $H$ is at least $k/(d^2+1)$.
By \autoref{the:deg3spa}, there exists a subgraph $H'$ of $H$ with maximum degree 3 and treewidth at least $k/((d^2+1) \log^\spc k)$.
We can assume that $V(H') = V(H)$.
Let $Q \subseteq V(G)$ be the vertices of $G$ that are not in any $B_v$.
The graph $H'$ is a minor of $G$, with a minor model whose branch sets are the sets $B_v$ for each $v \in I$ and singleton sets $\{u\}$ for each $u \in Q$.

We will construct a set $S \subseteq V(G)$ so that $Q \subseteq S$, all vertices of $I \cap S$ are sparsifiable in $G[S]$, and $H'$ has a minor model in $G$ whose branch sets are the sets $B_v \cap S$ for all $v \in I$ and singleton sets $\{u\}$ for all $u \in Q$.
The graph $G[S]$ therefore will contain $H'$ as a minor and therefore will have treewidth at least $k/((d^2+1) \log^\spc k)$.

Consider a vertex $v \in I$.
We will construct $B_v \cap S$ so that either $v \notin S$ or $v$ is sparsifiable in $G[S]$.
Because $H'$ has maximum degree 3, we can choose a set $T \subseteq B_v$ of at most three terminal vertices that are at distance 2 from $v$ whose connectivity should be preserved in $G[B_v \cap S]$ in order to preserve the minor model of $H'$.
We start by setting $B_v \cap S$ to be the union of the shortest paths from the terminals $t \in T$ to $v$.
Note that the only vertices at distance 2 from $v$ that are on the shortest paths are the terminals $T$.
Then, we say that a terminal $t \in T$ is private to a vertex $u \in N(v) \cap S$ if $u$ is the only vertex in $N(v) \cap S$ adjacent to $t$.
If a vertex $u \in N(v) \cap S$ does not have a private terminal, we remove $u$ from $S$.
Now, every vertex in $N(v) \cap S$ has a private terminal.
If $|N(v) \cap S| \le 2$, the vertex $v$ has degree at most 2 in $G[S]$ and we are done.
The remaining case is that $|N(v) \cap S| = 3$ and each vertex in $N(v) \cap S$ has a private terminal, implying that $|T|=3$ and the edges between $N(v) \cap S$ and $T$ form a matching.
If the graph $G[N(v) \cap S]$ is connected, we remove $v$ from $S$ and are done.
If the graph $G[N(v) \cap S]$ is not connected, it contains at most one edge.
If it contains no edges, the vertices $N(v) \cap S$ have degree 2 in $G[S]$, and therefore $v$ is a type~2 vertex in $G[S]$.
If it contains one edge, then $v$ is a type~3 vertex in $G[S]$.
\end{proof}

Next we finish the proof of \autoref{the:main} by applying \autoref{lem:main} $d^4+1$ times.

\begin{lemma}
\label{lem:end}
If a graph $G$ has treewidth $k$ and maximum degree at most $\log^{1/5} k$, then it contains an induced subgraph that is sparsifiable and has treewidth at least $k/2^{\OO(\log^{5/6} k)}$.
\end{lemma}
\begin{proof}
Let $G$ be a graph with treewidth $k$ and maximum degree $d$.
The vertices of $G$ can be partitioned into $d^4+1$ distance-5 independent sets $I_1, \ldots, I_{d^4+1}$ by observing that there are at most $d^4$ vertices at distance at most 4 from any vertex and using a greedy method.
We then sequentially apply \autoref{lem:main} with these distance-5 independent sets, in particular letting $G_0 = G$, and then for each $i$ with $1 \le i \le d^4+1$ letting $G_i$ be the graph obtained by applying \autoref{lem:main} with $G_{i-1}$ and $I_i \cap V(G_{i-1})$.
As taking induced subgraphs only increases distances, $I_i \cap V(G_{i-1})$ is a distance-5 independent set in $G_{i-1}$.
It also holds that once a vertex becomes sparsifiable, it will stay sparsifiable because taking induced subgraphs cannot increase vertex degrees or add new neighbors.
Therefore $G_{d^4+1}$ is sparsifiable.
The graph $G_{d^4+1}$ has treewidth at least \mbox{$k/((d^2+1) \log^\spc k)^{d^4+1}$}, meaning that when $d^5 \le \log k$, the decrease in the treewidth is by a factor of at most
\[((d^2+1) \log^\spc k)^{d^4+1} = 2^{\OO(d^4 (\log (d^2+1) + \log \log^\spc k))} = 2^{\OO(d^4 \log \log^\spc k)} = 2^{\OO(\log^{5/6} k)}.\]
\end{proof}

\autoref{the:main} follows from using \autoref{lem:end} to obtain a sparsifiable induced subgraph with treewidth at least $k/2^{\OO(\log^{5/6} k)}$, then applying the grid minor theorem~\cite{DBLP:journals/jctb/ChuzhoyT21} to obtain a $\Omega(k^{1/10}) \times \Omega(k^{1/10})$ grid minor, and then using \autoref{lem:spars} to argue that this grid minor, after contracting the corners, is also an induced minor.

\section{Proofs of corollaries}
We detail how \autoref{cor:main} and \autoref{cor:alg} follow from \autoref{the:main}.

\cormain*
\begin{proof}
The proof of Theorem 1.1 in~\cite{DBLP:journals/ejc/AboulkerAKST21} shows that there is a function $g : \mathbb{N} \rightarrow \mathbb{N}$ so that if a graph contains a triangulated $g(k) \times g(k)$-grid as an induced minor, then it contains either a $k \times k$-wall or the line graph of a $k \times k$-wall as an induced subgraph.
Then, this corollary follows from \autoref{the:main} by observing that a $k \times k$-grid contains a triangulated $k/6 \times k/6$-grid as an induced minor and setting $f_d(k) = c \cdot (g(k)^{10}+2^{d^5})$ for some large enough constant $c$.
\end{proof}

\coralg*
\begin{proof}
Let $G$ be the input graph and assume that $n$ is sufficiently large compared to $H$.
While there is a vertex of degree at least $\log^{1/5} (n/\log n)$ in $G$, we branch from this vertex.
This branching tree has size at most 
\[n {n \choose n / \log^{1/5} (n/\log n)}<n(e \cdot \log^{1/5}(n/\log n))^{n/(\log^{1/5}(n/\log n))} = 2^{\OO(n / \log^{1/6} n)}.\]
Then, if all vertices of $G$ have degree less than $\log^{1/5} (n/\log n)$ and $G$ has treewidth more than $c n/\log n$ for some constant $c$, then by \autoref{the:main} $G$ contains a $\Omega(n^{1/11}) \times \Omega(n^{1/11})$ grid induced minor, and therefore contains any planar graph $H$ as an induced minor if $n$ is sufficiently large compared to $H$.
Therefore, if all vertices of $G$ have degree less than $\log^{1/5} (n/\log n)$, then $G$ must have treewidth $\OO(n / \log n)$, and we can use a parameterized single-exponential time constant-factor approximation of treewidth~\cite{DBLP:conf/focs/Korhonen21} together with dynamic programming~\cite{DBLP:journals/dam/ArnborgP89} to solve maximum weight independent set in $2^{\OO(n / \log n)}$ time.
\end{proof}

\section*{Acknowledgements}
I thank Daniel Lokshtanov for telling me about the conjecture of Aboulker et al.

\bibliographystyle{plain}
\bibliography{paper}

\end{document}